\documentclass{IEEEtran}

\usepackage{cite}
\usepackage{amsmath,amssymb,amsfonts}
\usepackage{algorithmic}
\usepackage{graphicx}
\usepackage{textcomp}
\def\BibTeX{{\rm B\kern-.05em{\sc i\kern-.025em b}\kern-.08em
    T\kern-.1667em\lower.7ex\hbox{E}\kern-.125emX}}

\usepackage{amsthm}
\usepackage{mathtools}
\usepackage{tikz}
\usetikzlibrary{positioning}
\usepackage{pgfplots}
\usepgfplotslibrary{colormaps}
\usepackage{filecontents}

\newtheorem{theorem}{Theorem}

\newtheorem{lemma}{Lemma}
\newtheorem{definition}{Definition}

\newtheorem{assumption}{Assumption}
\newtheorem{remark}{Remark}
\newtheorem{example}{Example}

\newcommand{\Ne}{\mathbb{N}}
\renewcommand{\Re}{\mathbb{R}}

\newcommand{\vz}{\boldsymbol{z}}

\newcommand{\Alg}{\mathcal{A}}
\newcommand{\Blg}{\mathcal{B}}
\newcommand{\Prob}{\mathsf{P}}
\newcommand{\Exp}{\mathsf{E}}

\begin{document}

\title{Improved Compression-Based PAC Bounds\\for General Scenario Decision Making}

\author{Guillaume O.~Berger
\thanks{This project has received funding from the Belgian National Fund for Scientific Research (F.R.S.--FNRS).
GB is with ICTEAM institute, UCLouvain, Louvain-la-Neuve, Belgium.
He is a FNRS Postdoctoral Researcher.
(e-mails: guillaume.berger@uclouvain.be).}}

\maketitle

\begin{abstract}
Scenario decision making is a distribution-free approach to decision making in the presence of stochastic constraints, which consists in drawing samples of the random constraint (called ``scenarios'') and making a decision that satisfies all---or a predefined fraction---of the sampled constraints.
Several bounds have been proposed in the literature on the number of samples that is sufficient to guarantee with high probability that the risk of the scenario decision, i.e., its probability of violating a random constraint, is low.
We focus on sample-complexity bounds that depend on the compression size of the decision-making problem.
We propose new compression-based sample-complexity bounds that improve upon the existing ones without requiring stronger assumptions on the problem.
Specifically, our bounds do not rely on the so-called ``nondegeneracy'' assumption (which is widely used in the scenario decision-making literature, but can be limiting in practice); our bounds are better than the existing ones not relying on the nondegeneracy assumption, and is very close to the best known bounds relying on the nondegeneracy assumption.
\end{abstract}

\begin{IEEEkeywords}
Optimization, Randomized Algorithms, Statistical Learning.
\end{IEEEkeywords}

\section{Introduction}\label{sec:introduction}

\IEEEPARstart{R}{isk-aware} decision making is an important problem in many applications of engineering, such as control, energy planning, healthcare, etc.
In this context, the risk of a decision is defined as the probability that it violates some random constraint, e.g., the probability that the planned trajectory of an autonomous vehicle collides with some stochastically moving obstacle (other cars, pedestrians, etc.).
In many applications of risk-aware decision making, the probability distribution of the random constraint is unknown (e.g., in trajectory planning, or energy management).
This makes the problem of making decisions with a low risk challenging in general.
The approach of \emph{scenario decision making}~\cite{campi2008theexact,alamo2009randomized,campi2011asamplinganddiscarding,margellos2014ontheroad,esfahani2015performance,campi2018waitandjudge,campi2018ageneral,garatti2022risk,campi2023compression,rocchetta2024asurvey} provides an effective way to address this challenge by leveraging the principle of sample-based methods.
The approach consists in drawing $N$ samples (called \emph{scenarios}) of the random constraint, and making a decision that satisfies all, or a predefined fraction, of the sampled constraints.
The decision made from the scenarios is called the \emph{scenario decision}.
A key aspect of this approach is that under some assumptions on the problem, one can provide guarantees on the risk of the scenario decision.
In this work, we focus on \emph{PAC} (\emph{Probably Approximately Correct}) guarantees~\cite{shalevshwartz2014understanding,mohri2018foundations}: namely, upper bounds on the probability (with respect to the sampling of the $N$ scenarios) that the risk of the scenario decision exceeds some tolerance $\epsilon$.
These bounds, referred to as \emph{confidence bounds}, depend on $N$, $\epsilon$, and other problem-dependent quantities generally known as ``problem complexity''.
Common problem complexity measures include the VC dimension, the Rademacher complexity, and the compression size; see, e.g.,~\cite{rocchetta2024asurvey} for a survey.

\begin{figure}
    \centering
    \begin{tikzpicture}
        \pgfmathtruncatemacro{\N}{6};
        \pgfmathsetmacro{\dx}{2}
        \foreach \i in {1,...,\N}
        {
            \pgfmathsetmacro{\y}{2-2*(\i-1)/(\N-1)}
            \node (c-\i) at (0,\y) {$z_\i$};
        }
        \draw[rounded corners] (-0.2,-0.2) rectangle (0.2,2.2);
        \draw (\dx-0.5,0) rectangle (\dx+0.5,0.8);
        \node at (\dx,0.4) {Comp.};
        \draw[-latex] (0.2,0.6) -- (\dx-0.5,0.4);
        \pgfmathsetmacro{\dy}{2/(\N-1)}
        \foreach \j/\i in {1/2,2/5,3/6}
        {
            \pgfmathtruncatemacro{\k}{(\j-2)}
            \node (c-\j) at (2*\dx,0.4-\k*\dy) {$z_\i$};
        }
        \draw[rounded corners] (2*\dx-0.2,0.4-\dy-0.2) rectangle (2*\dx+0.2,0.4+\dy+0.2);
        \draw[-latex] (\dx+0.5,0.4) -- (2*\dx-0.2,0.4);
        \draw (3*\dx-0.5,0) rectangle (3*\dx+0.5,0.8);
        \node at (3*\dx,0.4) {Alg.};
        \draw[-latex] (2*\dx+0.2,0.4) -- (3*\dx-0.5,0.4);
        \node (dec) at (4*\dx,1) {$x$};
        \draw[-latex] (3*\dx+0.5,0.4) -- (dec);
        \draw (2*\dx-0.5,1.2) rectangle (2*\dx+0.5,2);
        \node at (2*\dx,1.6) {Alg.};
        \draw[-latex] (0.2,1.3) -- (2*\dx-0.5,1.6);
        \draw[-latex] (2*\dx+0.5,1.6) -- (dec);
    \end{tikzpicture}
    \caption{Scenario decision algorithm (Alg.) that has compression size $3$: any input set of scenarios (e.g., $z_1,\ldots,z_6$) can be reduced or compressed (Comp.) to a subset of size $3$ (e.g., $z_2,z_5,z_6$) and give the same decision ($x$) when given as input to Alg.}
    \label{fig:compression}
\end{figure}

In this paper, we focus on the compression size, introduced in 1986 by Littlestone and Warmuth~\cite{littlestone1986relating}.
The compression size of a scenario decision-making algorithm is the smallest number of scenarios to which any finite set of scenarios can be reduced while providing the same decision; see Figure~\ref{fig:compression} for an illustration.
In~\cite{littlestone1986relating}, a confidence bound is provided for scenario decision-making algorithms with bounded compression size.
A better confidence bound is provided in~\cite{campi2008theexact} under additional assumptions on the problem (among others, each constraint must define a convex set of admissible decisions, and the set of admissible decisions with respect to all constraints must have nonempty interior).
This bound is shown to be tight for certain instances of such problems (called \emph{fully-supported})~\cite{campi2008theexact}.
A similar confidence bound is proposed in~\cite{calafiore2010random} when the algorithm and the probability distribution of the random constraint satisfy the so-called \emph{nondegeneracy} assumption (essentially saying that the minimal-size compression set is unique).
This assumption has been widely used in the scenario decision making literature~\cite{campi2011asamplinganddiscarding,campi2018waitandjudge,garatti2022risk,romao2021tight,romao2023ontheexact}; however, it can be restrictive in practice (e.g., it excludes discrete probability distributions) and is by nature impossible to check when the probability distribution is unknown.

In the aforementioned works \cite{littlestone1986relating,campi2008theexact}, the scenario decision satisfies all the sampled constraints.
Therefore, we call these algorithms \emph{consistent}.
By contrast, in \emph{sample-and-discard} algorithms, introduced in~\cite{campi2011asamplinganddiscarding}, the scenario decision is allowed to violate some of the sampled constraints (which are ``discarded'').
A confidence bound for these algorithms under the nondegeneracy assumption and the so-called \emph{conformity} assumption (essentially saying that the discarded constraints are violated by the scenario decision) is proposed in~\cite{campi2011asamplinganddiscarding}.
A better bound (reminiscent of the one in~\cite{campi2008theexact}) is proposed in~\cite{romao2021tight} (see also \cite{romao2023ontheexact}) when additionally the solution of the sample-based problem is computed using so-called \emph{cascade optimization} (removing violated constraints sequentially, i.e., in ``cascade'').
This bound is shown to be tight for some instances (called \emph{fully-supported})~\cite{romao2023ontheexact}.
However, the use of cascade optimization can be restrictive, sub-optimal, and computationally expensive.
Furthermore, the analysis in~\cite{campi2011asamplinganddiscarding,romao2021tight,romao2023ontheexact} still requires the nondegeneracy assumption, which can be restrictive.
Finally, in~\cite{margellos2015ontheconnection}, a confidence bound that does not require the nondegeneracy assumption is proposed.
However, this bound is looser than the ones in~\cite{campi2011asamplinganddiscarding,romao2021tight,romao2023ontheexact}, and still assumes the conformity assumption.

The contribution of this paper is to provide new compres\-sion-based confidence bounds for consistent and sample-and-discard scenario decision-making algorithms, that do not require nondegeneracy, conformity, or similar assumptions on the problem and the algorithm.
These bounds are to the best of our knowledge the best available bounds in the literature without those assumptions.

\subsection*{Related work}

The initial derivation of the main result (Theorem~\ref{thm:our-bound-strong}) was strongly inspired by the recent work~\cite{garatti2021therisk}.
This work provides bounds on the risk of scenario decisions using the \emph{wait-and-judge} approach (initially introduced in~\cite{campi2018waitandjudge}).
In the wait-and-judge approach, the compression size needs not be bounded for all sets of scenarios.
However, given a set of $N$ scenarios, one can compute its compression size (which can be any integer between $0$ and $N$), and from it, derive a probabilistic upper bound on the risk.
Notably, the wait-and-judge bound in~\cite{garatti2021therisk} does not require additional assumptions on the problem or the algorithm.
When the compression size of every set of $N$ scenarios is upper bounded, then the wait-and-judge bound still applies, and can be turned into a compression-based confidence bound (Theorem~\ref{thm:bound-strong}, item 3).
Interestingly, the resulting confidence bound already improves upon the confidence bound proposed in\cite{littlestone1986relating}.
By looking carefully at the derivation of the wait-and-judge bound in~\cite{garatti2021therisk}, one can observe that some parameters of the proof can be optimized to provide a better confidence bound in the case the compression size of every set of $N$ scenarios is upper bounded.
The best set of parameters that we were able to find for this purpose lead to the confidence bound presented in our main result (Theorem~\ref{thm:our-bound-strong}).
The proof in~\cite{garatti2021therisk} relies on advanced technical results related to duality in semi-infinite optimization.
By contrast, the proof that we propose in this paper for our main result relies only on elementary probability theory.

\textit{Notation.}
We let $\Ne$ be the set of nonnegative integers.
For $n\in\Ne$, we let $[n]=\{1,\ldots,n\}$.
For a set $A$, we let $A^*=\bigcup_{m=0}^\infty A^m$.
For $m\in\Ne$ and a finite set $A$, the set of all subsets of $A$ with $m$ elements is denoted by $\binom{A}{m}$.

\section{Problem Statement and Background}\label{sec:problem-statement}

We introduce the problem of risk-aware scenario decision making.%
\footnote{The definitions and concepts presented in this section are classical, and can be found, e.g., in~\cite{alamo2009randomized,margellos2015ontheconnection,garatti2022risk,campi2023compression}.}
We assume given a set $X$ of \emph{decisions} and a set $Z\subseteq2^X$ of \emph{constraints} on $X$.
Given a decision $x\in X$ and a constraint $z\in Z$, we say that $x$ \emph{satisfies} $z$ if $x\in z$; otherwise, we say that $x$ \emph{violates} $z$.

\begin{example}\label{exa:path}
In the following optimization problem:
\[
\min_{x\in\Re^2}\; \lVert x\rVert \quad \text{s.t.} \quad a^\top (x - c) \leq 1 \;\; \forall\, a\in\Re^2,\:\lVert a\rVert\leq1,
\]
where $c\in\Re^2$ is fixed, the decision space is $X=\Re^2$, and the constraint space is $Z=\{X_a\subseteq X : a\in\Re^2,\:\lVert a\rVert\leq1\}$, where $X_a=\{x\in\Re^2 : a^\top (x - c)\leq 1\}$.
\end{example}

Finding a decision that satisfies all constraints $z$ in $Z$ is often intractable if $Z$ is large or unknown.
An approach to circumvent this---in the case where $Z$ can be sampled---is to sample $N$ constraints $z_1,\ldots,z_N$ from $Z$, and find a decision $x$ based on the sampled constraints.
This process is called a \emph{scenario decision algorithm}, which is thus a function from tuples of constraints to decisions:

\begin{definition}
A \emph{scenario decision algorithm} is a function that given a tuple of constraints $(z_1,\ldots,z_N)\in Z^*$ returns a decision $x\in X$.
Hence, it is a function $\Alg:Z^*\to X$.
\end{definition}

\begin{remark}
We consider \emph{tuples} of constraints, instead of \emph{sets} of constraints, because in Section~\ref{sec:weak-consistent} the multiplicity of the constraints will play a role.
Furthermore, we need to describe the probability measure of the sets of i.i.d.~sampled constraints, which is much easier to define when working with tuples of constraints (it is simply the product measure).
\end{remark}

\begin{example}\label{exa:path-algo}
Continuing Example~\ref{exa:path}, a scenario decision algorithm $\Alg$ can be defined as follows: given $z_i=X_{a_i}$, define $\Alg(z_1,\ldots,z_N)$ as the optimal solution of
\[
\min_{x\in\Re^2}\; \lVert x\rVert \quad \text{s.t.} \quad a_i^\top (x - c) \leq 1 \;\; \forall\, i\in[N].
\]
This is an instance of \emph{scenario optimization}~\cite{campi2008theexact}.
\end{example}

The output of a scenario decision algorithm is sometimes called the \emph{scenario decision}.
Since the scenario decision is obtained from a subset of the constraints in $Z$, one can generally not hope that it satisfies \emph{all} the constraints in $Z$.
However, one can hope that it satisfies all constraints in $Z$ except possibly those in a subset of small measure.
This property is formalized by the notion of risk:

\begin{definition}
Given a probability measure $\Prob$ on $Z$, and a decision $x\in X$, the \emph{risk} (also called \emph{violation probability}) of $x$ w.r.t.~$\Prob$, denoted by $V_\Prob(x)$, is the probability that $x$ violates a random constraint $z\in Z$, i.e., $V_\Prob(x) \triangleq \Prob[\{z\in Z : x\notin z\}]$.
\end{definition}

The risk of the scenario decision depends on the scenarios, which are sampled i.i.d.~at random according to the distribution of the random constraint.
Given a risk tolerance $\epsilon$ and a number of scenarios $N$, we aim to upper bound the probability that the risk of the scenario decision is below $\epsilon$:

\begin{definition}\label{def:pac-bound}
Consider a scenario decision algorithm $\Alg$.
A \emph{confidence bound} for $\Alg$ is a function $q:(0,1]\times\Ne\to[0,1]$ such that for any \emph{tolerance} $\epsilon\in(0,1)$, \emph{sample size} $N\in\Ne$ and probability distribution $\Prob$ on $Z$, it holds with probability $1-q(\epsilon,N)$ that if one samples $N$ scenarios $(z_1,\ldots,z_N)\in Z^N$ i.i.d.~according to $\Prob$, then the scenario decision returned by $\Alg$ has risk below $\epsilon$, i.e.,
\[
\Prob^N\!\left[ \left\{ \vz \in Z^N : V_\Prob(\Alg(\vz))>\epsilon \right\} \right] \leq q(\epsilon,N),
\]
where $\vz$ is a shorthand notation for $(z_1,\ldots,z_N)$.
\end{definition}

\begin{remark}
If for all $\epsilon\in(0,1]$, $\lim_{N\to\infty}q(\epsilon,N)=0$, then the scenario decision algorithm is said to be \emph{PAC} (\emph{Probably Approximately Correct})~\cite{valiant1984atheory,shalevshwartz2014understanding,mohri2018foundations}.
\end{remark}

Our goal in this paper is to provide confidence bounds for scenario decision algorithms whose compression size (see Section~\ref{ssec:compression} below) is bounded.
We make two assumptions on the scenario decision algorithms that we consider in this work: the algorithm is permutation invariant and stable:

\begin{assumption}\label{assum:well}
The scenario decision algorithm $\Alg$ is
\begin{itemize}
    \item \emph{permutation invariant}: for all $(z_1,\ldots,z_N)\in Z^*$ and all permutations $(i_1,\ldots,i_N)$ of $[N]$, $\Alg(z_1,\ldots,z_N)=\Alg(z_{i_1},\ldots,z_{i_N})$;
    \item \emph{stable}: for all $(z_1,\ldots,z_{N+1})\in Z^{\geq1}$, $\Alg(z_1,\ldots,z_N)\in z_{N+1}$ implies that $\Alg(z_1,\ldots,z_{N+1})=\Alg(z_1,\ldots,z_N)$.
\end{itemize}\vskip0pt
\end{assumption}

Those assumption are very mild and standard in the literature~\cite{campi2008theexact,alamo2009randomized,calafiore2010random,campi2011asamplinganddiscarding,margellos2015ontheconnection,romao2021tight,romao2023ontheexact,campi2023compression,rocchetta2024asurvey}.
In Sections~\ref{sec:strong-consistent} and~\ref{sec:weak-consistent}, we will consider separately two additional assumptions on the algorithm: (i) that the scenario decision satisfies all the sampled constraints (called \emph{consistent}), and (ii) that the scenario decision satisfies all but a fixed number $r$ of the sampled constraints (called \emph{sample-and-discard}).

\subsection{Compression Sets}\label{ssec:compression}

Before ending this section, let us recall the notion of \emph{compression}, which is instrumental in this work since we study confidence bounds based on the compression size.
This notion was introduced in 1986 by Littlestone and Warmuth~\cite{littlestone1986relating} in an unpublished manuscript cited in~\cite{floyd1995sample}.

The notion of compression describes the property that the input tuple of scenarios $\vz\in Z^*$ is compressible in the sense that a fixed-length subtuple $\vz'$ of $\vz$ leads to the same output:

\begin{definition}[Compression]\label{def:compression}
Given a scenario decision algorithm $\Alg$ and $\vz\coloneqq(z_1,\ldots,z_N)\in Z^*$, a subset $\{i_1,\ldots,i_M\}\subseteq[N]$, is a \emph{compression set} for $\vz$ and $\Alg$ if $\Alg(z_{i_1},\ldots,z_{i_M})=\Alg(\vz)$.
The set of all compression sets of $\vz$ and $\Alg$ is denoted by $\kappa(\vz;\Alg)$.\footnote{When clear from the context, we drop the dependence on $\Alg$ and simply write $\kappa(\vz)$.}
The \emph{compression size} of $\Alg$, denoted by $\rho(\Alg)$, is the infimum of all $d\in\Ne$ such that for every $\vz\in Z^*$, there is a compression set $I\in\kappa(\vz;\Alg)$ with $\lvert I\rvert\leq d$.
\end{definition}

We review confidence bounds available in the literature for scenario decision algorithms whose compression size is bounded.
We then propose new confidence bounds under Assumption~\ref{assum:well}.
We do this for consistent algorithms (Section~\ref{sec:strong-consistent}) and sample-and-discard algorithms (Section~\ref{sec:weak-consistent}).

\section{Consistent Algorithms}\label{sec:strong-consistent}

In this section, we focus on consistent scenario decision algorithms, i.e., algorithms for which the scenario decision satisfies all the sampled constraints:

\begin{definition}\label{def-strong-consistent}
A scenario decision algorithm $\Alg$ is \emph{consistent} if for all $\vz\coloneqq(z_1,\ldots,z_N)\in Z^*$, $\Alg(\vz)\in\bigcap_{i=1}^N z_i$.
\end{definition}

The algorithm in Example~\ref{exa:path-algo} is consistent.

\subsection{Previous Bounds in the Literature}

We review confidence bounds available in the literature for consistent algorithms:

\begin{theorem}\label{thm:bound-strong}
Consider a consistent scenario decision algorithm $\Alg$ satisfying Assumption~\ref{assum:well}.
Let $d\in\Ne$, and assume that $\rho(\Alg)\leq d$.
Then, confidence bounds for $\Alg$ are given by
\begin{enumerate}
    \item $q(N,\epsilon)=\binom{N}{d}(1-\epsilon)^{N-d}$, $N\in\Ne_{\geq d}$ and $\epsilon\in(0,1]$; see~\cite[Theorem~6]{floyd1995sample},~\cite[Theorem~2]{margellos2015ontheconnection};
    \item $q(N,\epsilon)=\sum_{i=0}^{d-1}\binom{N}{i}\epsilon^i(1-\epsilon)^{N-i}$, $N\in\Ne_{\geq d}$ and $\epsilon\in(0,1]$, under the ``nondegeneracy'' assumption~\cite[Definition~2.7]{calafiore2010random} on $\Alg$ and $\Prob$, or the ``convex feasible set with nonempty interior'' assumption~\cite[Assumption~1]{campi2008theexact} on $\Alg$; see~\cite[Theorem~3.3]{calafiore2010random},~\cite[Theorem~1]{campi2008theexact};
    \item $q(N,\epsilon)=\frac{N\binom{N}{d}(1-\epsilon)^{N-d}}{\sum_{m=d}^{N-1}\binom{m}{d}(1-\epsilon)^{m-d}}$, $N\in\Ne_{\geq d}$ and $\epsilon\in(0,1]$; derived from~\cite[Theorem~1]{garatti2021therisk}.
\end{enumerate}\vskip0pt
\end{theorem}

\begin{remark}\label{rem:not-pac-strong}
Note that $q(N,\epsilon)$ in item 2) is not a confidence bound \emph{stricto sensu} because it is not satisfied for every $\Prob$ (indeed, $\Prob$ needs to satisfy the nondegeneracy assumption); hence, it is not ``distribution-free''.
The nondegeneracy assumption is restrictive: it excludes distributions that have ``concentrated masses'' (aka.~``atoms'').
\end{remark}

\begin{remark}
The bound in item 3) is derived from the ``wait-and-judge'' bound in~\cite{garatti2021therisk}.\footnote{To be complete, let us mention that wait-and-judge bounds were proposed, e.g., in~\cite{campi2018waitandjudge,garatti2022risk}.
These bounds are better than the one in~\cite{garatti2021therisk}; however, they require the nondegeneracy assumption.
A key achievement of~\cite{garatti2021therisk} is to remove this assumption and obtain a bound that is almost as good.}
This provides an upper bound on the risk of the scenario decision, with predefined confidence, as a function of the confidence level, the number of samples, and \emph{the compression size of the tuple of sampled constraints}.
When the compression size of any tuple of sampled constraints is bounded by $d$, then a confidence bound in the sense of Definition~\ref{def:pac-bound} can be retrieved, which is presented in item 3).

As mentioned in the ``Related work'' section, our main result (Theorem~\ref{thm:our-bound-strong} below), which improves upon the bounds in items 1) and 3), was initially obtained by optimizing parameters in the proof of~\cite[Theorem~1]{garatti2021therisk}, to obtain better confidence bounds when the compression size of the algorithm is bounded.
Nevertheless, for the best set of parameters that we found, a simpler, more intuitive proof was possible, which is presented in the next subsection.
\end{remark}

Figure~\ref{fig:compare-strong} compares the different bounds from Theorem~\ref{thm:bound-strong}.
As we can see, the bound in item 2) is significantly better than the one in item 1); in fact, the former is tight for so-called ``fully-supported'' problems~\cite[Definition~2.5]{calafiore2010random}.
The bound in item 3) is much closer to the one in item 2) and does not require the nondegeneracy assumption.
In the next subsection, we provide a new compression-based confidence bound that is better than the one in item 3) and does not require the nondegeneracy assumption.
The proof is also elementary.\footnote{It does not rely on advanced mathematical tools such as Hausdorff moment problems as in \cite[Theorem~1]{campi2008theexact}, or duality in semi-infinite optimization as in \cite[Theorem~1]{garatti2021therisk}.}

\begin{figure}
    \centering
    \begin{tikzpicture}
        \begin{axis}[legend pos=south east,grid=major,xlabel={$d$},ylabel={$\epsilon$},xtick={0,100,...,500},minor x tick num=4,ytick={0,0.2,...,1},minor y tick num=4,grid=both,remember picture]
            \addplot+[mark=none,line width=1pt] table[x index=0,y index=1] {data_strong.txt};
            \addlegendentry{Theorem~\ref{thm:bound-strong}, 1)}
            \addplot+[mark=none,line width=1pt] table[x index=0,y index=2] {data_strong.txt};
            \addlegendentry{Theorem~\ref{thm:bound-strong}, 2)}
            \addplot+[mark=none,line width=1pt] table[x index=0,y index=3] {data_strong.txt};
            \addlegendentry{Theorem~\ref{thm:bound-strong}, 3)}
            \addplot+[mark=none,line width=1pt] table[x index=0,y index=4] {data_strong.txt};
            \addlegendentry{Theorem~\ref{thm:our-bound-strong}}
            \coordinate (inset) at (rel axis cs:0.01,0.99);
            \draw (axis cs:50,0.1) rectangle (axis cs:150,0.38);
            \coordinate (frame-nw) at (axis cs:50,0.38);
            \coordinate (frame-ne) at (axis cs:150,0.38);
        \end{axis}
        \begin{axis}[at={(inset)},anchor={outer north west},width=2.4cm,grid=major,xmin=50,xmax=150,ymin=0.1,ymax=0.38,axis background/.style={fill=white},ticks=none,scale only axis=true,
                     xtick={60,80,...,140},ytick={0.12,0.16,...,0.36},grid=both,remember picture]
            \addplot+[mark=none,line width=1pt] table[x index=0,y index=1] {data_strong.txt};
            \addplot+[mark=none,line width=1pt] table[x index=0,y index=2] {data_strong.txt};
            \addplot+[mark=none,line width=1pt] table[x index=0,y index=3] {data_strong.txt};
            \addplot+[mark=none,line width=1pt] table[x index=0,y index=4] {data_strong.txt};
            \coordinate (inset-sw) at (rel axis cs:0,0);
            \coordinate (inset-se) at (rel axis cs:1,0);
        \end{axis}
        \draw (frame-ne) -- (inset-se);
        \draw (frame-nw) -- (inset-sw);
    \end{tikzpicture}
    \caption{The curves show the value of $\epsilon$ so that $q(N,\epsilon)=0.05$ for $N=500$ for different values of $d$.
    Lower values of $\epsilon$ indicate better performance.
    We see that our bound (Theorem~\ref{thm:our-bound-strong}) is very close to the bound in item 2) without requiring the nondegeneracy assumption.
    It is also slightly better than the one in item 3) and significantly better than the one in item 1), which all have the same assumptions.}
    \label{fig:compare-strong}
\end{figure}

\subsection{Our Bound}\label{ssec:our-bound-strong}

We will prove the following result, which is the first main contribution of this paper:

\begin{theorem}\label{thm:our-bound-strong}
Consider a consistent scenario decision algorithm $\Alg$ satisfying Assumption~\ref{assum:well}.
Let $d\in\Ne$, and assume that $\rho(\Alg)\leq d$.
A confidence bound for $\Alg$ is given by
\[
q(N,\epsilon)=\binom{N}{d}\min_{m=d,\ldots,N}\binom{m}{d}^{-1}(1-\epsilon)^{N-m}
\]
for all $N\in\Ne_{\geq d}$ and $\epsilon\in(0,1]$.
\end{theorem}

Figure~\ref{fig:compare-strong} shows the bound in Theorem~\ref{thm:our-bound-strong}.
As we can see, this bound is better than the one in item 3) and requires the same assumptions.
It is also very close to the one in item 2).

\subsection{Proof of Theorem~\ref{thm:our-bound-strong}}

In the proof, we make the assumption that for each $\vz\in Z^{\geq d}$, there is a \emph{single} compression set of size $d$ for $\vz$.
This assumption is made \emph{without loss of generality} because one can always define an augmented decision set, augmented constraint set, and augmented scenario decision algorithm, for which the assumption holds and such that the augmented scenario decision matches with the scenario decision after projecting on the original decision set.
For such a construction, we refer the reader to~\cite[first paragraph of \S2]{garatti2021therisk}.\footnote{This construction requires only the stability property in Assumption~\ref{assum:well}.}

\begin{assumption}\label{assum:single}
The scenario decision algorithm $\Alg$ satisfies that for all $\vz\in Z^{\geq d}$, $\lvert\kappa(\vz)\rvert=1$.
\end{assumption}

A consequence of Assumption~\ref{assum:single} is the following:

\begin{lemma}\label{lem:rand-extract}
Let $\Alg$ be as in Theorem~\ref{thm:our-bound-strong} and satisfy Assumption~\ref{assum:single}.
Let $\Prob$ be a probability measure on $Z$.
For any $m\in\Ne_{\geq d}$ and $I\in\binom{[m]}{d}$, it holds that $\Prob^m(\{\vz\in Z^m : I\in\kappa(\vz)\})=\binom{m}{d}^{-1}$.
\end{lemma}

\begin{proof}
Fix $m\in\Ne_{\geq d}$.
For each $I\in\binom{[m]}{d}$, define $S_I=\{\vz\in Z^m : I\in\kappa(\vz)\}$.
The permutation invariance of $\Alg$ implies that $\Prob^m(S_I)$ is independent of $I$.
Furthermore, by Assumption~\ref{assum:single}, for any $I_1\neq I_2$, $S_{I_1}\cap S_{I_2}=\emptyset$.
Hence, $\Prob^m(S_I)=\binom{m}{d}^{-1}$.
\end{proof}

Next, we leverage the bound in Theorem~\ref{thm:bound-strong}, item 1), and extend it by incorporating the probability that a given set of $d$ elements is extracted from a set of $m$ samples (Lemma~\ref{lem:key-lemma-strong}):

\begin{lemma}\label{lem:claim-for-key-lem-strong}
Let $\Alg$ be as in Theorem~\ref{thm:our-bound-strong} and satisfy Assumption~\ref{assum:single}.
Let $N\in\Ne_{\geq d}$, $I\in\binom{[N]}{d}$ and $m\in\Ne_{\geq d}$.
Let $(j_1,\ldots,j_m)\in[N]^m$ be such that there is $I'\coloneqq\{i_1,\ldots,i_d\}\subseteq[m]$ satisfying $I=\{j_{i_1},\ldots,j_{i_d}\}$.
It holds that for all $\vz\coloneqq(z_1,\ldots,z_N)\in Z^N$, if $I\in\kappa(\vz)$, then $I'\in\kappa(\vz')$, where $\vz'=(z_{j_1},\ldots,z_{j_m})$.
\end{lemma}

\begin{proof}
Let $\vz\coloneqq(z_1,\ldots,z_N)\in Z^N$ be such that $I\in\kappa(\vz)$.
Let $I'\coloneqq\{i_1,\ldots,i_d\}\subseteq[m]$ be such that $I=\{j_{i_1},\ldots,j_{i_d}\}$.
Let $\vz''=(z_{j_{i_1}},\ldots,z_{j_{i_d}})$.
By definition of $\kappa$, we just need to show that $\Alg(\vz'')=\Alg(\vz')$.
Therefore, first observe that $\Alg(\vz'')=\Alg(\vz)$ since $I\in\kappa(\vz)$.
Then, since $\Alg$ is consistent, it follows that $\Alg(\vz'')\in\bigcap_{k=1}^N z_k$.
Hence, since $\Alg$ is stable and permutation invariant, we deduce that $\Alg(\vz')=\Alg(\vz'')$.
\end{proof}

\begin{lemma}\label{lem:key-lemma-strong}
Let $\Alg$ be as in Theorem~\ref{thm:our-bound-strong} and satisfy Assumption~\ref{assum:single}.
Let $N\in\Ne_{\geq d}$, $I\in\binom{[N]}{d}$, and $\epsilon\in(0,1]$.
For all $m\in\Ne\cap[d,N]$, it holds that $\Prob^N(\{\vz\in Z^N : I\in\kappa(\vz),\,V_\Prob(\Alg(\vz))>\epsilon\})\leq \binom{m}{d}^{-1} (1-\epsilon)^{N-m}$.
\end{lemma}

\begin{proof}
Let $m\in\Ne\cap[d,N]$, $(j_1,\ldots,j_N)$ be a permutation of $[N]$, and $I'\coloneqq\{i_1,\ldots,i_d\}\subseteq[m]$ be such that $I=\{j_{i_1},\ldots,j_{i_d}\}$.
Let $T'=\{\vz\in Z^m : I'\in\kappa(\vz),\, V_\Prob(\Alg(\vz))>\epsilon\}$.
By Lemma~\ref{lem:claim-for-key-lem-strong} and the consistency of $\Alg$, it holds that for all $\vz\in Z^m$, if $I\in\kappa(\vz)$, then $\vz'\in T'$ and $\Alg(\vz')\in\bigcap_{k=m+1}^N z_k$, where $\vz$ is a shorthand notation for $(z_1,\ldots,z_N)$, and $\vz'$ for $(z_{j_1},\ldots,z_{j_m})$.

Given $x\in X$, let $R(x)=\{(z_1,\ldots,z_{N-m})\in Z^{N-m} : x\in\bigcap_{k=1}^{N-m}z_k\}$.
Note that $\Prob^{N-m}(R(x)) = (1-V_\Prob(x))^{N-m}$.
Let $T=\{\vz\in Z^N : I\in\kappa(\vz),\,V_\Prob(\Alg(\vz))>\epsilon\}$.
It holds that
\[
\Prob^N(T) = \Exp_{\vz\sim \Prob^N}[\boldsymbol{1}_{R(\Alg(\vz'))}(\vz'') \, \boldsymbol{1}_{T'}(\vz')],
\]
where $\vz''$ is a shorthand for $(z_{j_{m+1}},\ldots,z_{j_N})$.
It follows that
\[
\Prob^N(T) \leq (1-\epsilon)^{N-m} \, \Prob^m(T') \leq (1-\epsilon)^{N-m} \binom{m}{d}^{-1},
\]
where we used Lemma~\ref{lem:rand-extract} to obtain the second inequality.
\end{proof}

We are now able to conclude the proof of Theorem~\ref{thm:our-bound-strong}:

\begin{proof}[Proof of Theorem~\ref{thm:our-bound-strong}]
Fix $N\in\Ne_{\geq d}$, $\Prob$ a probability distribution on $Z$, and $\epsilon\in(0,1]$.
Let $T=\{\vz\in Z^N : V_\Prob(\Alg(\vz))>\epsilon\}$.
For each $I\in\binom{[N]}{d}$, let $T_I=\{\vz\in T : I\in\kappa(\vz)\}$.
It holds that $T=\bigcup_{I\in\binom{[N]}{d}} T_I$.
Hence, by Lemma~\ref{lem:key-lemma-strong} and the union bound, we obtain that for every $m\in\Ne\cap[d,N]$,
\[
\Prob^N(T) \leq \sum_{I\in\binom{[N]}{d}} \Prob^N(T_I) \leq \binom{N}{d}\binom{m}{d}^{-1} (1-\epsilon)^{N-m}.
\]
This concludes the proof of the theorem.
\end{proof}

\begin{remark}
It is interesting to note that the proof of Lemma~\ref{lem:key-lemma-strong} uses the same idea as the proof of~\cite[Theorem~6]{floyd1995sample}, or~\cite[Theorem~2]{margellos2015ontheconnection}.
It is Assumption~\ref{assum:single}, which can be made without loss of generality for permutation invariant algorithms, and the stability property that allow (via Lemmas~\ref{lem:rand-extract} and~\ref{lem:claim-for-key-lem-strong}) to obtain the bound in Theorem~\ref{thm:our-bound-strong}, which improves upon them.
\end{remark}

\section{Relaxing the Consistency Assumption}\label{sec:weak-consistent}

In some applications, consistency can be too restrictive: it is sometimes preferable to discard a few sampled constraints before passing them to a consistent scenario decision algorithm if this can lead to scenario decisions with better performances (e.g., lower cost), especially when a lot of constraints are sampled.
The algorithm that discards a fixed number of constraints and keeps the other ones is called a selection algorithm:

\begin{definition}\label{def:sel-algo}
A function $\sigma:Z^*\to Z^*$ is a \emph{selection algorithm} with \emph{discarding size} $r\in\Ne$ if for every $\vz\coloneqq(z_1,\ldots,z_N)\in Z^*$, there are integers $1\leq i_1<\ldots<i_M\leq N$, with $M\geq N-r$, such that $\sigma(\vz)=(z_{i_1},\ldots,z_{i_M})$.
\end{definition}

The resulting scenario decision algorithm is called a sample-and-discard decision algorithm:

\begin{definition}\label{def:weak-algo}
A scenario decision algorithm $\Blg$ is a \emph{sample-and-discard} decision algorithm with \emph{discarding size} $r\in\Ne$ if there is (i) a consistent scenario decision algorithm $\Alg$ and (ii) a selection algorithm $\sigma$ with discarding size $r$, such that $\Blg=\Alg\circ\sigma$.
The pair $(\Alg,\sigma)$ is called a \emph{decomposition} of $\Blg$.
\end{definition}

\begin{example}\label{exa:path-algo-weak}
Continuing Examples~\ref{exa:path}--\ref{exa:path-algo}, a sample-and-discard decision algorithm $\Blg$ can be defined as follows: (i) $\Alg$ is as in Example~\ref{exa:path-algo}, and (ii) $\sigma$ selects among $(z_1,\ldots,z_N)\in Z^*$, with $z_i=X_{a_i}$, the $N-r$ constraints for which $\lVert a_i\rVert$ are the lowest (using a tie-breaking rule in case of ties).
\end{example}

\subsection{Previous Bounds in the Literature}

We review confidence bounds available in the literature for sample-and-discard decision algorithms:

\begin{theorem}\label{thm:bound-weak}
Consider a sample-and-discard decision algorithm $\Blg$ with discarding size $r\in\Ne$ and decomposition $(\Alg,\sigma)$, where $\Alg$ satisfies Assumption~\ref{assum:well}.
Let $d\in\Ne$, and assume that $\rho(\Alg)\leq d$.
Then, confidence bounds for $\Blg$ are given by
\begin{enumerate}
    \item $q(N,\epsilon)=\binom{N}{d}\sum_{i=0}^r\binom{N-d}{i}\epsilon^i(1-\epsilon)^{N-d-i}$, $N\in\Ne_{\geq r+d}$ and $\epsilon\in(0,1]$; see~\cite[Theorem~4]{margellos2015ontheconnection};
    \item $q(N,\epsilon)=\binom{r+d-1}{r}\sum_{i=0}^{r+d-1}\binom{N}{i}\epsilon^i(1-\epsilon)^{N-i}$, $N\in\Ne_{\geq r+d-1}$ and $\epsilon\in(0,1]$, under the ``nondegeneracy'' assumption~\cite[Assumption~2.1]{campi2011asamplinganddiscarding} and the ``conformity'' assumption~\cite[Assumption~2.2]{campi2011asamplinganddiscarding}; see~\cite[Theorem~2.1]{campi2011asamplinganddiscarding};
    \item $q(N,\epsilon)=\sum_{i=0}^{r+d-1}\binom{N}{i}\epsilon^i(1-\epsilon)^{N-i}$, $N\in\Ne_{\geq r+d-1}$ and $\epsilon\in(0,1]$, under the ``sequential nondegeneracy'' assumption \cite[Assumption~3]{romao2023ontheexact} and the form~\cite[(21)]{romao2023ontheexact} for $\Alg$; see~\cite[Theorem~4]{romao2023ontheexact}.
\end{enumerate}\vskip0pt
\end{theorem}

\begin{remark}
Interestingly, \cite[Theorem~4]{margellos2015ontheconnection} assumes a ``conformity'' assumption~\cite[Assumption~4, item 3)]{margellos2015ontheconnection}, although it is not used in the proof.
Therefore, we omit it in item 1).
\end{remark}

\begin{remark}\label{rem:not-pac-weak}
The same comment as in Remark~\ref{rem:not-pac-strong} applies for items 2) and 3).
\end{remark}

Figure~\ref{fig:compare-weak} compares the different bounds from Theorem~\ref{thm:bound-weak}
As we can see, the bound in item 3) is significantly better than those in items 1) and 2).
However, it requires stronger assumptions.
On the other hand, the bounds in items 1) and 2) are more or less identical, despite the fact that item 1) does not require the nondegeneracy assumption.
In the next subsection, we provide a new compression-based confidence bound that is better than the one in items 1) in some regimes and does not require the nondegeneracy or conformity assumption.
The proof is also elementary.

\begin{figure}
    \centering
    \begin{tikzpicture}
        \begin{axis}[legend pos=south east,grid=major,xlabel={$d$},ylabel={$\epsilon$}]
            \addplot+[mark=none,line width=1pt] table[x index=0,y index=2] {data_weak.txt};
            \addlegendentry{Theorem~\ref{thm:bound-weak}, 1)}
            \addplot+[mark=none,line width=1pt] table[x index=0,y index=1] {data_weak.txt};
            \addlegendentry{Theorem~\ref{thm:bound-weak}, 2)}
            \addplot+[mark=none,line width=1pt] table[x index=0,y index=3] {data_weak.txt};
            \addlegendentry{Theorem~\ref{thm:bound-weak}, 3)}
            \addplot+[mark=none,line width=1pt] table[x index=0,y index=4] {data_weak.txt};
            \addlegendentry{Theorem~\ref{thm:our-bound-weak}}
        \end{axis}
    \end{tikzpicture}
    \caption{The curves show the value of $\epsilon$ so that $q(N,\epsilon)=0.05$ for $N=500$ and $r=50$ for different values of $d$.
    Lower values of $\epsilon$ indicate better performance.
    We see that our bound (Theorem~\ref{thm:our-bound-weak}) is between the bounds in items 1) and 2) without requiring the nondegeneracy or conformity assumption.
    The bound in item 3) is significantly better but requires stronger assumptions.}
    \label{fig:compare-weak}
\end{figure}

\subsection{Our Bound}\label{ssec:our-bound-weak}

We will prove the following result, which is the second main contribution of this paper:

\begin{theorem}\label{thm:our-bound-weak}
Consider a sample-and-discard decision algorithm $\Blg$ with discarding size $r\in\Ne$ and decomposition $(\Alg,\sigma)$, where $\Alg$ satisfies Assumption~\ref{assum:well}.
Let $d\in\Ne$, and assume that $\rho(\Alg)\leq d$.
A confidence bound for $\Blg$ is given by
\[
q(N,\epsilon)=\binom{N}{r}\binom{N-r}{d}\min_{m=d,\ldots,N-r}\binom{m}{d}^{-1}(1-\epsilon)^{N-r-m},
\]
for all $N\in\Ne_{\geq r+d}$ and $\epsilon\in(0,1]$.
\end{theorem}

Figure~\ref{fig:compare-weak} shows the bound provided in Theorem~\ref{thm:our-bound-weak}.
As we can see, in some regimes (roughly, $d>r$), the bound is better than the one in item 1) and requires the same assumptions.

\subsection{Proof of Theorem~\ref{thm:our-bound-weak}}

The proof builds upon Theorem~\ref{thm:our-bound-strong} and elementary combinatorics:

\begin{proof}[Proof of Theorem~\ref{thm:our-bound-weak}]
Fix $N\in\Ne_{\geq d+r}$, $\Prob$ a probability distribution on $Z$ and $\epsilon\in(0,1]$.
Let $T=\{\vz\in Z^N : V_\Prob(\Alg(\vz))>\epsilon\}$.
Define $\Delta=\{(i_1,\ldots,i_{N-r})\in\Ne^{N-r} : 1\leq i_1<\ldots<i_{N-r}\leq N\}$, and for each $I\coloneqq(i_1,\ldots,i_{N-r})\in\Delta$, define $T_I = \{\vz\in T : \sigma(\vz)=\vz'\}$, where $\vz$ is a shorthand notation for $(z_1,\ldots,z_N)$, and $\vz'$ for $(z_{i_1},\ldots,z_{i_{N-r}})$.
We will show that for all $I\in\Delta$,
\begin{equation}\label{eq:bound-inter}
\Prob^N(T_I) \leq \binom{N-r}{d}\min_{m=d,\ldots,N-r}\binom{m}{d}^{-1}(1-\epsilon)^{N-r-m}.
\end{equation}
Therefore, fix $I\coloneqq(i_1,\ldots,i_{N-r})\in\Delta$.
Note that $T_I \subseteq \{\vz\in Z^N : V_\Prob(\Alg(\vz'))>\epsilon\}$, where $\vz$ is a shorthand notation for $(z_1,\ldots,z_N)$, and $\vz'$ for $(z_{i_1},\ldots,z_{i_{N-r}})$.
We deduce \eqref{eq:bound-inter} from Theorem~\ref{thm:our-bound-strong} and the fact that $\Prob^N(\{\vz\in Z^N : V_\Prob(\Alg(\vz'))>\epsilon\})=\Prob^{N-r}(\{\vz'\in Z^{N-r} : V_\Prob(\Alg(\vz'))>\epsilon\})$.
We obtain the final result by using the union bound: $\Prob^N(T) \leq \sum_{I\in\Delta} \Prob^N(T_I)$, and the observation that $\lvert\Delta\rvert=\binom{N}{r}$.
\end{proof}

\section{Conclusion}

We reviewed and compared different confidence bounds available in the literature for scenario decision making based on the compression size of the problem.
Some of these bounds are known to be optimal but require strong assumptions while others require weaker assumptions but are less good and are not known to be optimal.
We improved the previous results by providing new bounds that do not require the strongest assumptions and are better than the available bounds with the same assumptions.
For future work, we plan to improve these bounds even further or showing their optimality, we also plan to extend them to other approaches such as the ``wait-and-judge'' approach.

%



\bibliographystyle{IEEEtran}
\bibliography{myrefs}

\end{document}